\documentclass[11pt,a4paper, final, twoside]{article}
\usepackage{amsmath}
\usepackage{fancyhdr}
\usepackage{amsthm}
\usepackage{amsfonts}
\usepackage{amssymb}
\usepackage{amscd}
\usepackage{graphicx}
\usepackage{afterpage}
\usepackage[colorlinks=true, urlcolor=blue,  linkcolor=blue, citecolor=blue]{hyperref}
\newtheorem{theorem}{Theorem}[section]

\newtheorem{corollary}[theorem]{Corollary}

\newtheorem{definition}[theorem]{Definition}
\newtheorem{example}[theorem]{Example}

\newtheorem{lemma}[theorem]{Lemma}

\newtheorem{proposition}[theorem]{Proposition}
\newtheorem{remark}[theorem]{Remark}

\begin{document}
\begin{center}
{\Large \bf{Common fixed point theorems in $C^{*}$-algebra-valued b-metric spaces endowed with a graph and  applications}}
\end{center} \vspace{10mm}
\centerline{Bahman Moeini$^{*,a}$ and Arsalan Hojat Ansari$^{b}$}\ \\
\centerline{\footnotesize  $^{a}$Department of Mathematics, Hidaj Branch, Islamic Azad University, Hidaj, Iran}
\centerline{\footnotesize  $^{b}$Department of Mathematics, Karaj Branch, Islamic Azad University, Karaj, Iran}
\vspace{1mm}
\footnote{$^*$ Corresponding author}
\footnote{E-mail:moeini145523@gmail.com}
\footnote{E-mail:mathanalsisamir4@gmail.com}
\footnote{\emph{}} \afterpage{} \fancyhead{} \fancyfoot{}
\fancyhead[LE, RO]{\bf\thepage} \fancyhead[LO]{\small
$C^{*}$-algebra-valued $b$-metric spaces endowed with a graph and applications} \fancyhead[RE]{\small
B. Moeini and A.H. Ansari}

\begin{abstract}
In this paper, we equip a $C^{*}$-algebra-valued $b$-metric spaces with a graph $G=(V, E)$ and establish some common fixed point theorems. Also, some examples in support of our main results are provided. Finally, as applications, existence and uniqueness of solution for a type of operator equation and an integral equation are discussed.
\end{abstract}
\section{Introduction}
One of the main directions in obtaining possible generalizations of fixed point results
is introducing new types of spaces. In \cite{ZHLJ2}, Ma et al. introduced the concept of $C^{*}$-algebra-valued metric spaces. The
main idea consists in using the set of all positive elements of a unital $C^{*}$-algebra instead of the set of real numbers. These authors showed that if $(X, \mathbb{A}, d)$ be a complete $C^{*}$-algebra-valued
metric space and $T:X\rightarrow X$ is a contractive mapping, i.e., there exists an $A\in \mathbb{A}$ with
$\Vert A\Vert <1$ such that
\[
d(Tx, Ty)\preceq A^{*}d(x,y) A,\ \ (\forall x,y\in X).
\]
Then $T$ has a unique fixed point in $X$.
This line of research was continued in \cite{ZKSR, KPG, ZHLJ, DSTK, ASZ},
where several other fixed point results were obtained in the framework of $C^{*}$-algebra valued
metric, as well as (more general) $C^{*}$-algebra-valued $b$-metric spaces.

In recent investigations, the study of fixed point theory endowed with a graph occupies a prominent place in many aspects. Recently, Jachymski \cite{JJA} proved a sufficient condition for a self-mapping $f$ of a metric space $(X, d)$ to be a Picard operator and applied it to the Kelisky-Rivlin
theorem on iterates of the Bernstein operators on the space $C[0, 1]$. For more study in this work with a graph see \cite{ARS, AK, BBR, FBO, SKM}

In this paper, new $G$-contractions in $C^{*}$-algebra-valued $b$-metric space are introduced and some fixed point results are proved whenever the contraction condition holds for those pair of elements that form edges of the graph. Also, some examples to elaborate and illustrate of our results are constructed. Finally,
as applications, existence and uniqueness of solution for a type of operator equation and an integral equation are given
\section{Basic notions}
Recall that a Banach algebra $\mathbb{A}$ (over the field $\mathbb{C}$ of complex numbers)
is said to be a $C^{*}$-algebra if there is an involution $*$ in $\mathbb{A}$ (i.e., a mapping
$*:\mathbb{A}\rightarrow \mathbb{A}$ satisfying $a^{**}=a$ for each $a\in \mathbb{A}$)
such that, for all $a,b\in \mathbb{A}$ and $\lambda, \mu \in \mathbb{C}$, the following holds:
\begin{itemize}
\item[(i)] $(\lambda a+\mu b)^{*}=\bar \lambda a^{*}+\bar \mu b^{*}$;
\item[(ii)] $(ab)^{*}=b^{*}a^{*}$;
\item[(iii)] $\Vert a^{*}a \Vert=\Vert a \Vert^{2}$.
\end{itemize}
Note that, form $(iii)$, it easy follows that $\Vert a \Vert= \Vert a^{*} \Vert$ for each $a\in \mathbb{A}$.
Moreover, the pair $(\mathbb{A},*)$ is called a unital $*$-algebra if $\mathbb{A}$ contains the identity
element $1_{\mathbb{A}}$. A positive element of $\mathbb{A}$ is an element $a\in \mathbb{A}$
such that $a^{*}=a$ and its spectrum $\sigma(a)\subset \mathbb{R_{+}}$, where $\sigma(a)=\{\lambda\in \mathbb{R}: \lambda 1_{\mathbb{A}}-a \ \text{is noninvertible}\}$.
The set of all positive elements will be denoted by $\mathbb{A_{+}}$. Such elements allow us to
define a partial ordering '$\succeq$' on the elements of $\mathbb{A}$. That is,
\[
b\succeq a \ \ \text{if and only if} \ \ b-a\in\mathbb{A_{+}}.
\]
If $a\in \mathbb{A}$ is positive, then we write $a\succeq 0_{\mathbb{A}}$, where $0_{\mathbb{A}}$
is the zero element of $\mathbb{A}$. Each positive element $a$ of a $C^{*}$-algebra $\mathbb{A}$ has a
unique positive square root. From now on, by $\mathbb{A}$ we mean a unital $C^{*}$-algebra with
identity element $1_{\mathbb{A}}$. Further, $\mathbb{A_{+}}=\{a\in \mathbb{A}: a\succeq 0_{\mathbb{A}}\}$ and $(a^{*}a)^{\frac{1}{2}}=\vert a\vert$.
\begin{lemma}\label{l2.1} \cite{DRG}
Suppose that $\mathbb{A}$ is a unital $C^{*}$-algebra with a unit $1_{\mathbb{A}}$.
\begin{itemize}
\item[(1)] For any $x\in \mathbb{A_{+}}$, we have $x \preceq 1_{\mathbb{A}} \Leftrightarrow
 \Vert x\Vert \leq 1$.
\item[(2)] If $a \in \mathbb{A_{+}}$ with $\Vert a\Vert < \frac{1}{2}$,
then $1_{\mathbb{A}}- a $ is invertible and $\Vert a(1_{\mathbb{A}}-a)^{-1}\Vert< 1$.
\item[(3)] Suppose that $a, b \in \mathbb{A}$ with $a, b \succeq  0_{\mathbb{A}}$ and
$ab = ba$, then $ab \succeq 0_{\mathbb{A}}$.
\item[(4)] By $\mathbb{A'}$ we denote the set $\{a\in \mathbb{A}: ab=ba, \forall b\in \mathbb{A}\}$.
Let $a \in \mathbb{A'}$ if $b, c \in \mathbb {A}$ with
$b \succeq c \succeq 0_{\mathbb{A}}$, and $1_{\mathbb{A}}- a \in \mathbb{A'}$ is an
invertible operator, then
\[
(1_{\mathbb{A}}- a)^{-1}b \succeq (1_{\mathbb{A}}- a)^{-1}c.
\]
\end{itemize}
\end{lemma}
Notice that in a $C^{*}$-algebra, if $0_{\mathbb{A}} \preceq a, b$, one cannot conclude that
$0_{\mathbb{A}} \preceq ab$. For example, consider the $C^{*}$-algebra
$\mathbb{M}_{2}(\mathbb{C})$ and set
$a = \left (
\begin{array}{l}
3\ \ \ 2\\
2\ \ \ 3
\end{array}
\right )$,
$b = \left (
\begin{array}{l}
1\ \ \ -2\\
-2\ \ \ 4
\end{array}
\right ),$
then
 $ab = \left (
\begin{array}{l}
-1\ \ \ 2\\
-4\ \ \ 8
\end{array}
\right )$. Clearly $a,b \in \mathbb{M}_{2}(\mathbb{C})_{+}$, while $ab$ is not.
\begin{definition}\label{d2.2}
Let $X$ be a nonempty set, and $A\in \mathbb{A_{+}}$
such that $A \succeq 1_{\mathbb{A}}\ (\text{i.e.,}\  \Vert A\Vert\geq 1)$. Suppose the mapping $d : X \times X\rightarrow \mathbb{A_{+}}$ satisfies:
\begin{itemize}
\item[(1)]  $0_{\mathbb{A}}\preceq d(x, y)$ for all $x, y \in X$ and $d(x, y) = 0_{\mathbb{A}}
 \Leftrightarrow x = y$;
\item[(2)] $d(x, y) = d(y, x)$ for all $x, y \in X$;
\item[(3)] $d(x, y) \preceq A[d(x, z) + d(z, y)]$ for all $x, y, z \in X$.
\end{itemize}
Then $d$ is called a $C^{*}$-algebra-valued $b$-metric on $X$ and $(X,\mathbb{A}, d)$ is a $C^{*}$-algebra-valued $b$-metric space with coefficient $A$.
\end{definition}
\begin{example}\label{ex2.3}
Let $X=L^{\infty}(E)$ and $H=L^{2}(E)$, where $E$ is a lebesgue measurable set. By $L(H)$ we denote
the set of bounded linear operator on Hilbert space $H$. Clearly, $L(H)$ is a $C^{*}$-algebra with usual
norm. \\
Define $d_{b}:X\times X\rightarrow L(H)$ by
\[
d_{b}(f,g)=\pi_{\vert f-g\vert^{p}},\ \ \ (\forall f,g \in X, p\geq 1).
\]
Where $\pi_{h}:H\rightarrow H$ is the multiplication operator defined by
\[
\pi_{h}(\phi)=h.\phi,
\]
for $\phi \in H$. It can be obtained that $(X,L(H),d_{b})$ is a complete $C^{*}$-algebra-valued $b$-metric space with coefficient $\Vert A\Vert\geq 2^{p}$.
\end{example}
\begin{definition}\label{d2.4}
Let $(X,\mathbb{A}, d)$ be a $C^{*}$-algebra-valued $b$-metric space. Suppose that
$(x_{n}) \subset X$ and $x \in X$. If for any $\epsilon > 0$  there is a positive integer $N$ such that for all
$n \geq N$, $\Vert d(x_{n}, x) \Vert\leq \epsilon$, then $(x_{n})$ is said
to be converge with respect to $\mathbb{A}$, and $(x_{n})$ converges to $x$ and $x$ is the limit of
$x_{n}$. We denote it by $\lim_{n\rightarrow\infty} x_{n} = x$.
If for any $\epsilon > 0$ there is a positive integer $N$ such that for all $n,m \geq N$, $\Vert d(x_{n}, x_{m})\Vert \leq \epsilon$, then $(x_{n})$ is called a Cauchy sequence with respect to $\mathbb{A}$.
We say $(X,\mathbb{A}, d)$ is a complete $C^{*}$-algebra-valued $b$-metric space if every Cauchy sequence with respect to $\mathbb{A}$ is convergent.
\end{definition}
\begin{remark} \label{r2.5}
If in Definition \ref{d2.2}, $A = 1_{\mathbb{A}}$, then the ordinary triangle inequality condition in a $C^{*}$-algebra valued metric space is satisfied. Thus a $C^{*}$-algebra-valued $b$-metric space is an ordinary $C^{*}$-algebra-valued metric space, but in general, a $C^{*}$-algebra-valued metric space is not necessary a $C^{*}$-algebra-valued $b$-metric space, (see Example $2.1$ of \cite{ZHLJ}).
In particular, if $\mathbb{A} = \mathbb{C}$ and $A = 1$, the $C^{*}$-algebra-valued $b$-metric spaces are just the ordinary metric spaces.
\end{remark}

Let $S$ and $T$ be self-mappings of a set $X$. Recall that, if $y = Sx = Tx$ for some $x \in X$, then $x$ is called a coincidence point of $S$ and $T$ and $y$ is called a point of coincidence of $S$ and $T$, and if $y=x$ then $x$ is called a common fixed point of $S$ and $T$. Suppose $C(S,T)$ and $PC(S,T)$ denote the sets of all coincidence point and point of coincidence, respectively, of the pair $\{S, T\}$. The mappings $S, T$ are weakly compatible \cite{GJ}, if for
every $x \in X$, the following holds:
\[
T(Sx) = S(Tx) \ \text{whenever}\ Sx = Tx.
\]
\begin{proposition} \label{p2.6} \cite{MAGJ} Let $S$ and $T$ be weakly compatible self-mappings of a nonempty
set $X$. If $S$ and $T$ have a unique point of coincidence $y = Sx = Tx$, then $y$ is
the unique common fixed point of $S$ and $T$.
\end{proposition}
Here we review some basic notions in graph theory.\\
Let $(X,\mathbb{A}, d)$ be a $C^{*}$-algebra-valued $b$-metric space.
Let $\Delta$ denote the diagonal of the Cartesian product $X \times X$. Consider a directed graph $G$ such that the set $V (G)$ of its vertices coincides with $X$, and the set $E(G)$ of its edges contains all loops,
i.e., $E (G) \supseteq \Delta$. We assume $G$ has no parallel edges, so we can identify $G$
with the pair $(V (G), E (G))$. Moreover, we may treat $G$ as a weighted graph by assigning to each edge the distance between its vertices.
By $G^{-1}$ we denote the conversion of a graph $G$, i.e., the graph obtained from
$G$ by reversing the direction of edges. Thus we have
\[
E(G^{-1})= \{(x, y)\in X \times X  :  (y, x) \in E(G)\}.
\]
The letter $\tilde{G}$ denotes the undirected graph obtained from $G$ by ignoring the
direction of edges. Actually, it will be more convenient for us to treat $\tilde{G}$
as a directed graph for which the set of its edges is symmetric. Under this convention,
\[
E(\tilde{G})= E (G) \cup E(G^{-1}).
\]
We call $(V ′, E′)$ a subgraph of $G$ if $V ′ \subseteq V (G) , E′ \subseteq E (G)$ and for any edge
$(x, y) \in E′,\  x, y \in V ′$.
Now we recall a few basic notions concerning the connectivity of graphs.
All of them can be found, e.g., in \cite{JHN}. If $x$ and $y$ are vertices in a graph $G$, then
a path in $G$ from $x$ to $y$ of length $N \ (N \in \mathbb{N})$ is a sequence $(x_{i})^{N}_{i=0}$
of $N + 1$ vertices such that $x_{0} = x, x_{N} = y$ and $(x_{i-1}, x_{i}) \in E (G)$ for $i = 1, \cdots ,N$. A graph $G$ is connected if there is a path between any two vertices. $G$ is weakly connected if $\tilde{G}$
is connected.
\section{Main results}
In this section, we assume that $(X, \mathbb{A}, d)$ is a $C^{*}$-algebra-valued $b$-metric space
with the coefficient $A\succeq 1_{\mathbb{A}}$, and $G$ is a graph such that $V (G) = X$ and $G$ has no parallel edges. Let $f, g : X \rightarrow X$ be such that $f(X) \subseteq g(X)$. If $x_{0} \in X$ is arbitrary, then there exists an element $x_{1} \in X$ such that $fx_{0} = gx_{1}$, since $f(X) \subseteq g(X)$. Proceeding in this way, we can construct a sequence $(gx_{n})$ such that $gx_{n} = fx_{n-1}, n = 1, 2, 3,
\cdots .$ By $C_{gf}$ we denote the set of all elements $x_{0}$ of $X$ such that $(gx_{n}, gx_{m}) \in E(\tilde{G})$ for $m, n = 0, 1, 2, \cdots $. If $g = I$, the identity map on $X$, then obviously $C_{gf}$ becomes $C_{f}$ which is the collection of all elements $x$ of $X$ such that $(f^{n}x, f^{m}x) \in E(\tilde{G})$ for $m, n = 0, 1, 2, \cdots $.\\
Let $PC(f,g)$ denotes the set of all points of coincidence of the pair $\{f,g\}$. In the following
some properties of the graph $G$ are given such that we will obtain useful results by using these conditions:\newline
$(P1)$ If $(gx_{n})$ is a sequence in $X$ such that $gx_{n} \rightarrow x$ and $(gx_{n}, gx_{n+1}) \in E(\tilde{G})$ for all $n \geq 1$, then there exists a subsequence $(gx_{n_{i}})$ of $(gx_{n})$ such that $(gx_{n_{i}}, x) \in E(\tilde{G})$ for all $i \geq 1$.\newline
$(P2)$ If $x, y \in PC(f,g)$, then $(x, y) \in E(\tilde{G})$.\newline
$(P3)$ If $(x_{n})$ is a sequence in $X$ such that $x_{n} \rightarrow x$ and $(x_{n+1}, x_{n}) \in E(\tilde{G})$ for all $n \geq 1$, then there exists a subsequence $(x_{n_{i}})$ of $(x_{n})$ such that $(x_{n_{i}}, x) \in E(\tilde{G})$ for all $i \geq 1$.\newline
$(P4)$ If $x, y$ are fixed points of $f$ in $X$, then $(x, y) \in E(\tilde{G})$.
\begin{theorem}\label{t3.1}
Let $(X, \mathbb{A}, d)$ be a $C^{*}$-algebra-valued $b$-metric space endowed with a graph $G=(V,E)$ and the mappings $f, g : X \rightarrow X$ satisfy
\begin{equation}\label{eq3.1}
d(fx, fy) \preceq B^{*} d(gx, gy) B,
\end{equation}
for all $x, y \in X$ with $(gx, gy) \in E(\tilde{G})$, where,  $B\in \mathbb{A}$, and $\Vert A\Vert \Vert B\Vert^{2} < 1$. Suppose $f(X) \subseteq g(X)$ and $g(X)$ is a complete subspace of $X$ with the property $(P1)$. Then the set $PC(f, g)\neq \emptyset$  if $C_{gf}\neq \emptyset$. Moreover, $PC(f, g)$ is singleton if the graph $G$ has the  property $(P2)$. Furthermore, if $f$ and $g$ are weakly compatible, then $f$ and $g$ have a unique common fixed point in $X$.
\end{theorem}
\begin{proof}
Suppose that $C_{gf}\neq \emptyset$. We choose an $x_{0}\in C_{gf}$ and keep it fixed. Since
$f(X) \subseteq g(X)$, there exists a sequence $(gx_{n})$ such that $gx_{n} = fx_{n-1}$, $n = 1, 2, 3, \cdots$. and $(gx_{n}, gx_{m}) \in E(\tilde{G})$ for $m, n = 0, 1, 2, \cdots$.
We now show that $(gx_{n})$ is a Cauchy sequence in $g(X)$. For any natural number $n$, we have by using condition (\ref{eq3.1}) that
\begin{equation}\label{eq3.2}
\begin{array}{rl}
d(gx_{n}, gx_{n+1}) &= d(fx_{n-1}, fx_{n}) \\
&\preceq B ^{*} d(gx_{n-1}, gx_{n}) B\\
&\preceq (B ^{*})^{2} d(gx_{n-2}, gx_{n-1}) (B)^{2}\\
&\preceq \\
&\vdots\\
&\preceq (B ^{*})^{n} d(gx_{0}, gx_{1}) (B)^{n},
\end{array}
\end{equation} for all $n\in \mathbb{N}$. Let us denote $ d(gx_{0}, gx_{1})$ by $Q\in \mathbb{A}$. Then
\begin{equation}\label{eq3.3}
d(gx_{n}, gx_{n+1}) \preceq (B ^{*})^{n} Q (B)^{n},
\end{equation}
for all $n\in \mathbb{N}$. For $m, n \in \mathbb{N}$ with $m > n$, by using conditions (\ref{eq3.3}), and $(3)$ of Definition \ref{d2.2}, we have
\[
\begin{array}{rl}
d(gx_{n}, gx_{m}) &\preceq Ad(gx_{n}, gx_{n+1}) + A^{2}d(gx_{n+1}, gx_{n+2})\\\\
&+ \cdots+ A^{m-n-1}d(gx_{m-2}, gx_{m-1}) + A^{m-n-1}d(gx_{m-1}, gx_{m})\\\\
&\preceq A (B ^{*})^{n} Q (B)^{n} + A^{2} (B ^{*})^{n+1} Q (B)^{n+1}\\\\
&+ \cdots+ A^{m-n-1} (B ^{*})^{m-2} Q (B)^{m-2} + A^{m-n-1} (B ^{*})^{m-1} Q (B)^{m-1}\\\\
&=A\Big[  (B ^{*})^{n} Q (B)^{n} + A (B ^{*})^{n+1} Q (B)^{n+1}\\\\
&+ \cdots+ A^{m-n-2} (B ^{*})^{m-2} Q (B)^{m-2}\Big] + A^{m-n-1} (B ^{*})^{m-1} Q (B)^{m-1}\\\\
&=A\sum_{j=n}^{m-2}\Big(A^{j-n} (B ^{*})^{j} Q (B)^{j}\Big) + A^{m-n-1} (B ^{*})^{m-1} Q (B)^{m-1}\\\\
&=A\sum_{j=n}^{m-2}\Big(A^{j-n} (B ^{*})^{j} Q^{\frac{1}{2}} Q^{\frac{1}{2}}(B)^{j}\Big) + A^{m-n-1} (B ^{*})^{m-1} Q^{\frac{1}{2}} Q^{\frac{1}{2}} (B)^{m-1}\\\\
&=A\sum_{j=n}^{m-2}\Big(A^{j-n} (Q^{\frac{1}{2}}B ^{j})^{*} (Q^{\frac{1}{2}}B^{j})\Big) + A^{m-n-1} (Q^{\frac{1}{2}}B ^{m-1})^{*}  (Q^{\frac{1}{2}} B^{m-1})\\\\
&=A\sum_{j=n}^{m-2}\Big(A^{j-n} \Big \vert Q^{\frac{1}{2}}B ^{j}\Big\vert^{2}\Big) + A^{m-n-1} \Big \vert Q^{\frac{1}{2}}B ^{m-1}\Big \vert^{2}\\\\
&\preceq \Big\Vert A\sum_{j=n}^{m-2}A^{j-n} \Big \vert Q^{\frac{1}{2}}B ^{j}\Big\vert^{2}\Big\Vert 1_{\mathbb{A}}+ \Big\Vert A^{m-n-1} \Big \vert Q^{\frac{1}{2}}B ^{m-1}\Big \vert^{2}\Big\Vert 1_{\mathbb{A}}\\\\
&\preceq \Vert A \Vert \sum_{j=n}^{m-2}\Vert A^{j-n} \Vert \Vert Q^{\frac{1}{2}}\Vert^{2}\Vert B^{j}\Vert^{2}1_{\mathbb{A}}+
\Vert A^{m-n-1}\Vert \Vert Q^{\frac{1}{2}}\Vert^{2} \Vert B^{m-1}\Vert^{2} 1_{\mathbb{A}}\\\\
&\preceq  \Vert A \Vert \sum_{j=n}^{m-2}\Vert A \Vert^{j-n} \Vert Q^{\frac{1}{2}}\Vert^{2}\Vert B^{j}\Vert^{2}1_{\mathbb{A}}+
\Vert A\Vert^{m-n-1} \Vert Q^{\frac{1}{2}}\Vert^{2} \Vert B^{m-1}\Vert^{2} 1_{\mathbb{A}}\\\\
&\preceq  \Vert A \Vert^{1-n}\Vert Q^{\frac{1}{2}}\Vert^{2} \sum_{j=n}^{m-2}\Vert A \Vert^{j} \Vert B^{2}\Vert^{j}1_{\mathbb{A}}+
\Vert A\Vert^{-n} \Vert A\Vert^{m-1}\Vert Q^{\frac{1}{2}}\Vert^{2} \Vert B^{2}\Vert^{m-1} 1_{\mathbb{A}}\\\\
&\preceq \Vert A \Vert^{1-n}\Vert Q^{\frac{1}{2}}\Vert^{2} \Big(\sum_{j=n}^{m-2}(\Vert A \Vert \Vert B^{2}\Vert)^{j}1_{\mathbb{A}}\Big)+
\Vert A\Vert^{-n} \Vert Q^{\frac{1}{2}}\Vert^{2}(\Vert A\Vert \Vert B^{2}\Vert)^{m-1}
1_{\mathbb{A}}\\\\
&\rightarrow 0_{\mathbb{A}} \ \ \text{as}\ \ m,n\rightarrow \infty,
\end{array}
\]
since the summation in the first term is a geometric series, and $\Vert A\Vert \Vert B^{2}\Vert<1$ implies that both $(\Vert A\Vert \Vert B^{2}\Vert)^{m-1}\rightarrow  0_{\mathbb{A}}$ and
 $(\Vert A\Vert \Vert B^{2}\Vert)^{n-1}\rightarrow  0_{\mathbb{A}}$. Therefore, $(gx_{n})$ is a Cauchy
 sequence in $g(X)$ with respect to $\mathbb{A}$. As $g(X)$ is complete, there exists an $u\in g(X)$ such that $gx_{n}\rightarrow u=gv$ for some $v\in X$. As $x_{0} \in C_{gf}$, it follows that $(gx_{n}, gx_{n+1}) \in E(\tilde{G})$ for all $n \geq 0$, and so by property $(P1)$, there exists a subsequence $(gx_{n_{i}} )$ of $(gx_{n})$ such that $(gx_{n_{i}}, gv) \in E(\tilde{G})$ for all $i \geq 1$.  By using conditions (\ref{eq3.1}), and $(3)$ of Definition \ref{d2.2}, we conclude that
\[
\begin{array}{rl}
d(fv,gv)&\preceq Ad(fv,fx_{n_{i}}) + Ad(fx_{n_{i}},gv)\\\\
&\preceq A B^{*}d(gv,gx_{n_{i}})B + Ad(gx_{n_{i+1}},gv)\\\\
&\rightarrow 0_{\mathbb{A}} \ \ \text{as}\ \ i\rightarrow \infty.
\end{array}
\]
This gives that $d(fv,gv)=0_{\mathbb{A}}$ and hence, $fv=gv=u$, i.e., $u\in PC(f,g)\neq\emptyset$.
Now if there is another point $u^{*}\neq u$ and $u^{*}\in PC(f,g)$ such that $fx=gx=u^{*}$ for some
$x\in X$. By property $(P2)$, we have $(u,u^{*})\in E(\tilde{G})$. Then
\[
d(u,u^{*})=d(fv,fx)\preceq B^{*}d(gv,gx) B= B^{*}d(u,u^{*}) B,
\]
since $\Vert B\Vert <1 $, then the above inequality yields that
\[
\Vert d(u,u^{*})\Vert\leq \Vert B\Vert^{2}\Vert d(u,u^{*}) \Vert<\Vert d(u,u^{*})\Vert.
\]
Which is a contradiction. Thus, $u=u^{*}$, i.e., the set $PC(f,g)$ is singleton.
If $f$ and $g$ are weakly compatible, then by Proposition \ref{p2.6}, $f$ and $g$ have a unique
common fixed point in $X$.
\end{proof}
\begin{example}\label{ex3.2}
Let $X=\mathbb{R}$ and consider, $\mathbb{A} = M_{2}(\mathbb{R})$, of all $2 \times 2$ matrices with the usual operation of addition, scalar multiplication, and matrix multiplication. Define norm on $\mathbb{A}$ by $\Vert A\Vert =\Big(\sum^{2}_{i,j=1}\vert a_{i,j}\vert^{2}\Big)^{\frac{1}{2}}$, and $*: M_{2}(\mathbb{R}) \rightarrow M_{2}(\mathbb{R})$, given by $A^{*}=A$, for all $A\in \mathbb{A}$,
defines a convolution on $M_{2}(\mathbb{R})$. Thus $\mathbb{A}$ becomes a $C^{*}$-algebra. For
\[
A = \left (
\begin{array}{l}
a_{11}\ \ \ a_{12}\\
a_{21}\ \ \ a_{22}
\end{array}
\right ),
B = \left (
\begin{array}{l}
b_{11}\ \ \ b_{12}\\
b_{21}\ \ \ b_{22}
\end{array}
\right ) \in M_{2}(\mathbb{R}) ,
\]
we denote $A\preceq B$ if and only if $(a_{ij}-b_{ij})\leq 0$, for all $i,j=1,2$. \\
Define $d: \mathbb{R} \times \mathbb{R}\rightarrow M_{2}(\mathbb{R})$ by
\[
d(x,y) = \left (
\begin{array}{l}
\vert x-y\vert^{2}\ \ \ \ \ \ 0\\
\ \  \ 0\ \ \ \ \ \ \ \ \vert x-y\vert^{2}
\end{array}
\right ).
\]
It is easy to check that $d$ satisfies all the conditions of Definition \ref{d2.2}, for proving $(3)$ we only need to use the following inequality:
\[
\vert x-y\vert^{p}\leq 2^{p}(\vert x-z\vert^{p}+\vert z-y\vert^{p}),
\]
for all $p>1$. So, $(\mathbb{R},  M_{2}(\mathbb{R}), d)$ is a complete $C^{*}$-algebra-valued
$b$-metric space with the coefficient
$A=\left (
\begin{array}{l}
4\ \ \ 0\\
0\ \ \ 4
\end{array}
\right )$.
Let $G$ be a graph such that $V(G)=\mathbb{R}$ and $E(G)=\Delta \cup \Big\{\Big(0, \frac{1}{3^{n}}\Big): n=0,1,2, \cdots\Big\}$. Define $f,g: \mathbb{R}\rightarrow \mathbb{R}$ by
\[
f(x) = \left \{
\begin{array}{l}
\frac{x}{3},\ \ \ \ \  if \ x\neq \frac{1}{3},\\
1,\ \ \ \  \ if \ x=\frac{1}{3},
\end{array}
\right .
\]
and $gx=2x$ for all $x\in \mathbb{R}$. Obviously, $f(X)\subseteq g(X)=X=\mathbb{R}$. If
$x,y\in \mathbb{R}$ with $(gx, gy)\in E(\tilde{G})$, then there exist two cases:\newline
Case $1$. $gx=gy$.\newline
Then $x=y$, $d(fx,fy)=\left (
\begin{array}{l}
0\ \ \ 0\\
0\ \ \ 0
\end{array}
\right )$
and $d(gx,gy)=\left (
\begin{array}{l}
0\ \ \ 0\\
0\ \ \ 0
\end{array}
\right )$.
So, for this case we have
\[
d(fx, fy) \preceq B^{*} d(gx, gy) B,
\]
Case $2$. $gx=0, gy=\frac{1}{3^{n}}$, $n=0,1,2,\cdots$.\newline
Then $x=0, y=\frac{1}{2.3^{n}}$,  $n=0,1,2,\cdots$, $d(fx,fy)=\left (
\begin{array}{l}
\frac{1}{4.3^{2n+2}}\ \ \  \ 0\\
\ \ 0\ \  \ \ \ \frac{1}{4.3^{2n+2}}
\end{array}
\right )$
and $d(gx,gy)=\left (
\begin{array}{l}
\frac{1}{3^{2n}}\ \ \ \ \ 0\\
\ 0\ \ \ \ \ \frac{1}{3^{2n}}
\end{array}
\right )$.
So, for this case for $n=0,1,2,\cdots$, we have
\[
\left (
\begin{array}{l}
\frac{1}{4.3^{2n+2}}\ \ \  \ 0\\
\ \ 0\ \  \ \ \ \frac{1}{4.3^{2n+2}}
\end{array}
\right )\preceq \left (
\begin{array}{l}
\frac{1}{4}\ \ \ \ \ 0\\
0\ \ \ \ \ \frac{1}{4}
\end{array}
\right )\left (
\begin{array}{l}
\frac{1}{3^{2n}}\ \ \ \ \ 0\\
\ 0\ \ \ \ \ \frac{1}{3^{2n}}
\end{array}
\right )\left (
\begin{array}{l}
\frac{1}{4}\ \ \ \ \ 0\\
0\ \ \ \ \ \frac{1}{4}
\end{array}
\right ),
\]
i.e.,
\[
d(fx, fy) \preceq B^{*} d(gx, gy) B,
\]
where, $B=\left (
\begin{array}{l}
\frac{1}{4}\ \ \ \ \ 0\\
0\ \ \ \ \ \frac{1}{4}
\end{array}
\right )
$
and $\Vert A\Vert\Vert B\Vert^{2}<1$.\\

Therefore, for all $x,y\in \mathbb{R}$ with $(gx, gy)\in E(\tilde{G})$,
\[
d(fx, fy) \preceq B^{*} d(gx, gy) B,
\]
where, $B=\left (
\begin{array}{l}
\frac{1}{4}\ \ \ \ \ 0\\
0\ \ \ \ \ \frac{1}{4}
\end{array}
\right )
$
and $\Vert A\Vert\Vert B\Vert^{2}<1$. We can verify that $0\in C_{gf}\neq\emptyset$. In fact
$gx_{n}=fx_{n-1}$, $n=1,2,3,\cdots$, gives that $gx_{1}=f0=0\Rightarrow x_{1}=0$ and so
$gx_{2}=fx_{1}=0\Rightarrow x_{2}=0$. Proceeding in this way we get $gx_{n}=0$ for
$n=0,1,2,\cdots$, and hence $(gx_{n}, gx_{m})\in E(\tilde{G})$ for $m,n=0,1,2,\cdots$.
Also if $(gx_{n})$ be a sequence in $X$ with the property $(gx_{n}, gx_{n+1})\in E(\tilde{G})$ then
$(gx_{n})$ must be either constant sequence or a sequence of the following form
\[
gx_{n} = \left \{
\begin{array}{l}
0,\ \ \ \ \ \ \  if \ n \ \text{is odd},\\
\frac{1}{3^{n}},\ \ \ \  \ if \ n\  \text{is even},
\end{array}
\right .
\]
where the words "odd'' and ''even'' are interchangeable. So it follows that property $(P1)$. Moreover
the graph $G$ has the property $(P2)$ and $f$ and $g$ are weakly compatible. Thus all the conditions of Theorem \ref{t3.1} are satisfied and $0$ is the unique common fixed point of $f$ and $g$ in $X$.
\end{example}
In the following we show that the weak compatibility condition in Theorem \ref{t3.1} cannot be relaxed.
\begin{remark}\label{r3.3}
 In Example \ref{ex3.2} if $g:X\rightarrow X$ is defined by $gx=2x-5$ for all $x\in X$,
instead of $gx=2x$, then $3\in C_{gf}\neq\emptyset$. In fact
$gx_{n}=fx_{n-1}$, $n=1,2,3,\cdots$, gives that $gx_{1}=fx_{0}=f3=1\Rightarrow x_{1}=3$ and so
$gx_{2}=fx_{1}=1\Rightarrow x_{2}=3$. Proceeding in this way we get $gx_{n}=1$ for
$n=0,1,2,\cdots$, and hence $(gx_{n}, gx_{m})=(1,1)\in E(\tilde{G})$ for $m,n=0,1,2,\cdots$.
Also, $f3=g3=1$ but $g(f3)\neq f(g3)$, i.e., $f$ and $g$ are not weakly compatible. However
all the other conditions of Theorem \ref{t3.1} are satisfied. We observe that the set $PC(f,g)$ is
singleton and $1\in PC(f,g)\neq \emptyset$ without being a common fixed point.
\end{remark}
\begin{corollary}\label{c3.4}
Let $(X, \mathbb{A}, d)$ be a complete $C^{*}$-algebra-valued $b$-metric space endowed with a graph $G=(V,E)$ and the mapping $f: X \rightarrow X$ be such that
\begin{equation}\label{eq3.4}
d(fx, fy) \preceq B^{*} d(x, y) B,
\end{equation}
for all $x, y \in X$ with $(x, y) \in E(\tilde{G})$, where, $B\in \mathbb{A}$, and $\Vert A\Vert \Vert B\Vert^{2} < 1$. Suppose  the quadratic $(X, \mathbb{A}, d, G)$ have the property $(P3)$. Then $f$  has a fixed point in $X$ if $C_{f}\neq \emptyset$. Moreover, $f$  has a unique fixed point in $X$ if the graph $G$ has the  property $(P4)$.
\end{corollary}
\begin{proof}
The proof can be obtained from Theorem \ref{t3.1} by considering $g=I$, the identity map on $X$.
\end{proof}
\begin{theorem}\label{t3.5}
Let $(X, \mathbb{A}, d)$ be a $C^{*}$-algebra-valued $b$-metric space endowed with a graph $G=(V,E)$ and the mappings $f, g : X \rightarrow X$ satisfy
\begin{equation}\label{eq3.5}
d(fx, fy) \preceq B \Big( d(fx, gx) + d(fy, gy) \Big),
\end{equation}
for all $x, y \in X$ with $(gx, gy) \in E(\tilde{G})$, where,  $B\in \mathbb{A'_{+}}$, and $\Vert BA\Vert \ <\frac{1}{2}$. Suppose $f(X) \subseteq g(X)$ and $g(X)$ is a complete subspace of $X$ with the property $(P1)$. Then the set $PC(f, g)\neq \emptyset$  if $C_{gf}\neq \emptyset$. Moreover, $PC(f, g)$ is singleton if the graph $G$ has the  property $(P2)$. Furthermore, if $f$ and $g$ are weakly compatible, then $f$ and $g$ have a unique common fixed point in $X$.
\end{theorem}
\begin{proof} As in the proof of Theorem \ref{t3.1}, we can construct a sequence $(gx_{n})$ in $X$ such
that $gx_{n} = fx_{n-1}$, $n = 1, 2, 3, \cdots$, and $(gx_{n}, gx_{m}) \in E(\tilde{G})$ for $m, n = 0, 1, 2, \cdots$. We shall show that $(gx_{n})$ is a Cauchy sequence in $g(X)$.\\
Without loss of generality, one can suppose that $B\neq 0_{\mathbb{A}}$. Notice that $B\in \mathbb{A'_{+}}$, $B \Big( d(fx, gx) + d(fy, gy) \Big)$ is also a positive element. For any natural number $n$, we have by using condition (\ref{eq3.5}) that
\begin{align*}
d(gx_{n+1}, gx_{n})=d(fx_{n}, fx_{n-1})&\preceq B\Big(d(fx_{n}, gx_{n}) + d(fx_{n-1}, gx_{n-1})\Big)\\
&=Bd(gx_{n+1}, gx_{n}) + Bd(gx_{n}, gx_{n-1}).
\end{align*}
Thus, $$(1_{\mathbb{A}}-B)d(gx_{n+1}, gx_{n})\preceq Bd(gx_{n}, gx_{n-1}).$$
Since $B\in \mathbb{A'_{+}}$ with $\Vert AB\Vert < \frac{1}{2}$, (so, $\Vert B\Vert<\frac{1}{2}$) one have $(1_{\mathbb{A}}-B)^{-1}\in \mathbb{A'_{+}}$ and furthermore $B(1_{\mathbb{A}}-B)^{-1}\in \mathbb{A'_{+}}$ with $\Vert B(1_{\mathbb{A}}-B)^{-1}\Vert< 1$ by Lemma \ref{l2.1}. Therefore,
\begin{equation}\label{eq3.6}
d(gx_{n+1}, gx_{n})\preceq B(1_{\mathbb{A}}-B)^{-1}d(gx_{n}, gx_{n-1})=td(gx_{n}, gx_{n-1}),
\end{equation}
where $t= B(1_{\mathbb{A}}-B)^{-1}$. By repeat use of condition (\ref{eq3.6}), we have
\begin{equation}\label{eq3.7}
d(gx_{n+1}, gx_{n})\preceq t^{n}d(gx_{1}, gx_{0})=t^{n}Q,
\end{equation}
for all $n\in \mathbb{N}$, where $Q=d(gx_{1}, gx_{0})\in \mathbb{A}$.
For any $m\geq 1$ and $p\geq 1$, it follows that
\[
\begin{array}{rl}
d(gx_{m+p}, gx_{m}) &\preceq A \Big[ d(gx_{m+p}, gx_{m+p-1}) + d(gx_{m+p-1}, gx_{m})\Big]\\\\
&=Ad(gx_{m+p}, gx_{m+p-1}) + Ad(gx_{m+p-1}, gx_{m})\\\\
&\preceq Ad(gx_{m+p}, gx_{m+p-1}) +A^{2}\Big[d(gx_{m+p-1}, gx_{m+p-2}) \\\\
&+ d(gx_{m+p-2}, gx_{m})\Big]\\\\
&\preceq\\\\
&\vdots\\\\
&\preceq Ad(gx_{m+p}, gx_{m+p-1}) + A^{2}d(gx_{m+p-1}, gx_{m+p-2})\\\\
&+ \cdots+ A^{p-1}d(gx_{m+2}, gx_{m+1}) + A^{p-1}d(gx_{m+1}, gx_{m})\\\\
&\preceq A t^{m+p-1} Q + A^{2} t^{m+p-2} Q + \cdots+ A^{p-1} t^{m+1} Q + A^{p-1} t^{m} Q\\\\
&=\sum_{j=1}^{p-1}\Big(A^{j} t^{m+p-j} Q \Big) + A^{p-1} t^{m} Q \\\\
&=\sum_{j=1}^{p-1}\Big \vert A^{\frac{j}{2}} t^{\frac{m+p-j}{2}} Q^{\frac{1}{2}}\Big \vert^{2}  + \Big \vert A^{\frac{p-1}{2}} t^{\frac{m}{2}} Q^{\frac{1}{2}}\Big \vert^{2} \\\\
&\preceq \Vert Q \Vert \sum_{j=1}^{p-1} \Vert A\Vert^{j}\Vert t\Vert^{m+p-j}1_{\mathbb{A}} +
\Vert A\Vert^{p-1}\Vert t\Vert^{m} \Vert Q \Vert 1_{\mathbb{A}}\\\\
&\preceq \frac{\Vert A\Vert^{p}\Vert t\Vert^{m+1} \Vert Q \Vert }{\Vert A\Vert-\Vert t\Vert}1_{\mathbb{A}} + \Vert A\Vert^{p-1}\Vert t\Vert^{m} \Vert Q \Vert 1_{\mathbb{A}}\\\\
&\rightarrow 0_{\mathbb{A}} \ \ \text{as}\ m\rightarrow \infty.
\end{array}
\]
This implies that $(gx_{n})$ is a Cauchy sequence in $g(X)$ with respect to $\mathbb{A}$. As $g(X)$ is complete, there exists an $u\in g(X)$ such that $gx_{n} \rightarrow u = gv$ for some $v \in X$. As
$x_{0}\in  C_{gf}$, it follows that $(gx_{n},gx_{n+1}) \in E(\tilde{G})$ for all $n \geq 0$, and so by property (P1), there exists a subsequence $(gx_{n_{i}})$ of $(gx_{n})$ such that $(gx_{n_{i}},gv) \in  E(\tilde {G})$ for all $i \geq 1$. Now using conditions \ref{eq3.5} and  \ref{eq3.7}, we obtain
\[
\begin{array}{rl}
d(fv, gv)&\preceq A d(fv, fx_{n_{i}}) +  A d(fx_{n_{i}}, gv)\\\\
&\preceq \Big[ AB d(fv,gv) + AB d(fx_{n_{i}}, gx_{n_{i}})\Big] + A d(gx_{n_{i+1}}, gv),
\end{array}
\]
this is equivalent to
\[
(1_{\mathbb{A}}-AB) d(fv, gv)\preceq AB d(gx_{n_{i+1}}, gx_{n_{i}}) +A d(gx_{n_{i+1}}, gv),
\]
so,
\[
\begin{array}{rl}
d(fv, gv)&\preceq (1_{\mathbb{A}}-AB)^{-1}ABd(gx_{n_{i+1}}, gx_{n_{i}}) + (1_{\mathbb{A}}-AB)^{-1}Ad(gx_{n_{i+1}}, gv)\\\\
&\preceq  (1_{\mathbb{A}}-AB)^{-1}AB t^{n_{i}}Q +  (1_{\mathbb{A}}-AB)^{-1}Ad(gx_{n_{i+1}}, gv)\\\\
&\preceq \Vert (1_{\mathbb{A}}-AB)^{-1}AB\Vert  \Vert t\Vert^{n_{i}}\Vert Q\Vert 1_{\mathbb{A}} +  \Vert (1_{\mathbb{A}}-AB)^{-1}A\Vert \Vert d(gx_{n_{i+1}}, gv)\Vert 1_{\mathbb{A}}\\\\
&\rightarrow 0_{\mathbb{A}} \ \ \text{as} \ i \rightarrow\infty.
\end{array}
\]
This gives that, $fv=gv=u$, therefore, $u\in PC(f,g)\neq \emptyset$. To prove the uniqueness of the
$PC(f,g)$, suppose that there is another point $u^{*}\in X$ such that $u^{*}\in PC(f,g)$, and
$u^{*}=fx=gx$ for some $x\in X$. By property $(P2)$, we have $(u, u^{*})\in E(\tilde{G})$. Then
\[
d(u, u^{*})=d(fv,fx)\preceq  B \Big( d(fv, gv) + d(fx, gx) \Big)=0_{\mathbb{A}},
\]
which gives that $u=u^{*}$, thus $PC(f,g)$ is singleton. If $f$ and $g$ are weakly compatible,
then by Proposition \ref{p2.6}, $f$ and $g$ have a unique common fixed point in $X$.
\end{proof}
\begin{example}\label{ex3.6}
Let $X=[0, +\infty)$, $\mathbb{A}$ and $d: X\times X\rightarrow \mathbb{A}$ are defined as in
Example \ref{ex3.2}. Let $G$ be a graph such that $V(G)=X$ and
$E(G)=\Delta \cup \{(3^{t}x,3^{t}(x+1)): x\in X: \  \text{with}\ x\geq 2, t=0,1,2,\cdots \}$. Suppose $f,g:X\rightarrow X$ be defined by
$fx=3x$ and $gx=9x$ for all $x\in X$. Clearly, $f(X)=g(X)=X$. If $x,y\in X$ with $(gx,gy)\in E(\tilde{G})$ then there exist two cases:\newline
Case $1$. $gx=gy$. \newline
Then $x=y$, $d(fx,fy)=\left (
\begin{array}{l}
0\ \ \ 0\\
0\ \ \ 0
\end{array}
\right )$. So, obviously
\[
d(fx, fy) \preceq B \Big( d(fx, gx) + d(fy, gy) \Big),
\]
for $B\in \mathbb{A'_{+}}$ and $\Vert BA\Vert \ <\frac{1}{2}$.\\

Case $2$. $gx=3^{t}z, gy=3^{t}(z+1), z\geq 2, t=0,1,2,\cdots$.\\
Then $x=3^{t-2}z, y=3^{t-2}(z+1)$ for all $z\geq 2, t=0,1,2,\cdots$, and we have
\[
\begin{array}{rl}
d(fx,fy)&=\left (
\begin{array}{l}
3^{2t-2}\ \ \ 0\\
\ 0\ \ \ 3^{2t-2}
\end{array}
\right )\\\\
&\preceq \left (
\begin{array}{l}
\frac{1}{52}\ \ \ 0\\
0\ \ \ \frac{1}{52}
\end{array}
\right )
\left (
\begin{array}{l}
3^{2t-2}(8z^{2}+8z+4)\ \ \ \ \ \ \ \ \ \ \ \ \ 0\\
\ \ \ \ \ \ \ \ \ \ \ \  0\ \ \ \ \ \ \ \ \ \ \ 3^{2t-2}(8z^{2}+8z+4)
\end{array}
\right )\\\\
&= \left (
\begin{array}{l}
\frac{1}{52}\ \ \ 0\\
0\ \ \ \frac{1}{52}
\end{array}
\right )
\Big[d(3^{t-1}z, 3^{t}z) + d(3^{t-1}(z+1), 3^{t}(z+1))\Big]\\\\
&= B \Big( d(fx, gx) + d(fy, gy) \Big),
\end{array}
\]
where, $B=\left (
\begin{array}{l}
\frac{1}{52}\ \ \ 0\\
0\ \ \ \frac{1}{52}
\end{array}
\right )\in \mathbb{A'_{+}}$ and $\Vert BA\Vert \ <\frac{1}{2}$, so condition (\ref{eq3.5}) is satisfied.
It can be verify that $0\in C_{gf}$. Also, any sequence $(gx_{n})$ with $gx_{n}\rightarrow x$ and
$(gx_{n}, gx_{n+1})\in E(\tilde{G})$ must be a constant sequence and hence property $(P1)$ holds. Furthermore, the graph $G$ has the property $(P2)$ and $f$ and $g$ are weakly compatible. Thus all the conditions of Theorem \ref{t3.5} are satisfied and $0$ is the unique common fixed point of $f$ and $g$ in $X$.
\end{example}
In the following, a fixed point result of $G$-Kannan operators in $C^{*}$-algebra-valued $b$-metric spaces is given
\begin{corollary}\label{c3.7}
Let $(X, \mathbb{A}, d)$ be a complete $C^{*}$-algebra-valued $b$-metric space endowed with a graph $G=(V,E)$ and the mapping $f: X \rightarrow X$ be such that
\begin{equation}\label{eq3.8}
d(fx, fy) \preceq B\Big(d(fx, x) + d(fy,y)\Big),
\end{equation}
for all $x, y \in X$ with $(x, y) \in E(\tilde{G})$, where, $B\in \mathbb{A'_{+}}$, and $\Vert AB\Vert < 1$. Suppose  the quadratic $(X, \mathbb{A}, d, G)$ have the property $(P3)$. Then $f$  has a fixed point in $X$ if $C_{f}\neq \emptyset$. Moreover, $f$  has a unique fixed point in $X$ if the graph $G$ has the  property $(P4)$.
\end{corollary}
\begin{proof}
The proof can be obtained from Theorem \ref{t3.5} by putting $g=I$, the identity map on $X$.
\end{proof}
\section{Applications}
In this section, as applications of contractive mappings theorem on complete $C^{*}$-algebra-valued $b$-metric spaces endowed with a graph, existence and uniqueness results for a type of operator equation
and integral equation are provided.
\begin{example}\label{ex4.1}
Suppose that $H$ is a Hilbert space, $L(H)$ is the set of linear bounded operators on $H$. Let
$B_{1},B_{2}, \cdots ,B_{n} \in L(H)$, which satisfy $\sum_{n=1}^{\infty}\Vert B_{n}\Vert^{4}< \frac{1}{4}$ and $X \in L(H), Q \in L(H)_{+}$, if $L(H)$ endowed with a graph $G$ such that for all $X,Y\in L(H)$,
$(X,Y)\in E(\tilde{G})$. Then the operator equation
\[
X-\sum_{n=1}^{\infty}B^{*}_{n}X B_{n}=Q
\]
has a unique solution in $L(H)$ if $C_{\sum_{n=1}^{\infty}B^{*}_{n}X B_{n}+Q}\neq\emptyset$.
\end{example}
\begin{proof}
Set $\beta=\sum_{n=1}^{\infty}\Vert B_{n}\Vert^{2}$ . Clear that if $\beta= 0$, then the $B_{n} = 0_{L(H)} (n \in \mathbb{N})$, and the equation has a unique solution in $L(H)$. Without loss of generality, one can suppose that $\beta > 0$. Choose a positive operator $T \in L(H)$. For $X,Y \in L(H)$, set
\[
d(X,Y) = \Vert X - Y\Vert^{2}T.
\]
Then it is easy to verify that $d(X,Y)$ is a $C^{*}$-algebra-valued $b$-metric space with
coefficient $A=2^{2}1_{L(H)}$ and $(L(H), d)$ is complete since $L(H)$ is a Banach space.
Consider the map $F : L(H)\rightarrow L(H)$ defined by
\[
F(X)=\sum_{n=1}^{\infty}B^{*}_{n}XB_{n}+Q
\]
Then
\[
\begin{array}{rl}
d(F(X), F(Y))&= \Vert F(X) -F(Y)\Vert^{2}T\\\\
&= \Vert\sum_{n=1}^{\infty} B^{*}_{n} (X - Y)B_{n}\Vert^{2}T\\\\
&\preceq \sum_{n=1}^{\infty}\Vert B_{n}\Vert^{4}\Vert X - Y\Vert^{2}T\\\\
&\preceq \beta^{2}d(X,Y)\\\\
&=(\beta 1_{{L(H)}})^{*}d(X,Y)(\beta 1_{{L(H)}}),
\end{array}
\]
so,
\[
d(F(X), F(Y))\preceq B^{*} d(X,Y) B,
\]
for $B=\beta 1_{{L(H)}}$, and $\Vert A\Vert \Vert B\Vert^{2}<2^{2}\sum_{n=1}^{\infty}\Vert B_{n}\Vert^{4}< 1$. Since for all $X,Y\in L(H)$, $(X,Y)\in E(\tilde{G})$, then properties $(P3)$ and $(P4)$ satisfy for $(L(H), G, d)$. Also by hypothessis $C_{F}=C_{\sum_{n=1}^{\infty}B^{*}_{n}XB_{n}+Q}\neq\emptyset$,
thus by using Corollary \ref{c3.4}, there exists a unique fixed point $X$ in $L(H)$. Furthermore, since
$\sum_{n=1}^{\infty}B^{*}_{n}XB_{n}+Q$ is a positive operator, the solution is a Hermitian operator.
\end{proof}
\begin{example}\label{ex4.1}
Let $E$ be a Lebesgue measurable set, $X=L^{\infty}(E)$ and $H=L^{2}(E)$ be the Hilbert space.
Consider the integral equation
\begin{equation}\label{eq4.9}
x(t)=\int_{E}k(t,s,x(s))ds + g(t),\ \ t\in E.
\end{equation}
Suppose that
\begin{itemize}
\item[(a)] $k:E\times E\times \mathbb{R}\rightarrow \mathbb{R}$ and $g\in L^{\infty}(E)$;
\item[(b)] there exists a continuous function $\varphi: E\times E\rightarrow \mathbb{R}$ and $\beta\in (0,\frac{1}{\sqrt{2^{p}}}), (p\geq 1)$ such that
\[
\vert k(t, s, u(s))-k(t, s, v(s))\vert\leq \beta \vert \varphi(t, s)(u(s)-v(s)) \vert,
\]
for $t, s\in E$ and $u, v\in \mathbb{R}$;
\item[(c)] $\sup_{t\in E}\int_{E}\vert \varphi (t, s)\vert ds\leq 1$;\\

For $S,T\in X$, define $d_{b}:X\times X\rightarrow L(H)$ by
\[
d_{b}(T,S)=\pi_{\vert T-S\vert^{p}},\ \ \ ( p\geq 1),
\]
where $\pi_{h}:H\rightarrow H$ be defined as Example \ref{ex2.3}.
Then $(X, L(H), d_{b})$ is a complete $C^{*}$-algebra-valued $b$-metric space with coefficient $\Vert A\Vert= 2^{p}$.
\item[(d)]  the space $(X, L(H), d)$ endowed with a graph $G$ such that
for all $x,y\in X$, $(x,y)\in E(\tilde{G})$.
\end{itemize}
Then the integral equation (\ref{eq4.9}) has a unique solution $x^{*}$ in $L^{\infty}(E)$ if $C_{\int_{E}k(t,s,x(s))ds + g(t)}\neq \emptyset$.
\end{example}
\begin{proof}
Let $T: L^{\infty}(E) \rightarrow L^{\infty}(E)$ be
\[
T(x(t))=\int_{E}k(t,s,x(s))ds + g(t),\ \ t\in E.
\]
Set $B=\beta 1_{L(H)}$, then $B\in L(H)_{+}$ and $\Vert B \Vert=\beta< \frac{1}{\sqrt{2^{p}}}$. For
any $h\in H$
\[
\begin{array}{rl}
\Vert d_{b}(Tx,Ty)\Vert &=\sup_{\Vert h \Vert=1}(\pi_{\vert Tx-Ty \vert^{p}}h, h)\\\\
&=\sup_{\Vert h \Vert=1}\int_{E}\Big[ \Big\vert \int_{E}(k(t,s,x(s))-k(t,s,y(s)))ds \Big\vert^{p}\Big]h(t)\overline{h(t)}dt\\\\
&\leq \sup_{\Vert h \Vert=1}\int_{E}\Big[ \Big\vert \int_{E}(k(t,s,x(s))-k(t,s,y(s)))ds \Big\vert^{p}\Big]\vert h(t)\vert^{2}dt\\\\
&\leq \sup_{\Vert h \Vert=1}\int_{E}\beta^{p}\Big[ \int_{E}(\varphi(t,s)(x(s)-y(s)))ds \Big]^{p}\vert h(t)\vert^{2}dt\\\\
&\leq \beta^{2}\sup_{\Vert h \Vert=1}\int_{E}\Big[ \int_{E}(\vert \varphi(t,s)\vert ds \Big]^{p}\vert h(t)\vert^{2}dt. \Vert (x-y)^{p} \Vert_{\infty}\\\\
&\leq \beta^{2}\sup_{t\in E}\int_{E}\vert \varphi(t,s)\vert^{p} ds.\sup_{\Vert h \Vert=1}\int_{E}\vert h(t)\vert^{2}dt. \Vert (x-y)^{p} \Vert_{\infty}\\\\
&\leq \beta^{2 }\Vert (x-y)^{p} \Vert_{\infty}\\\\
&=\Vert B\Vert^{2}\Vert d_{b}(x,y)\Vert.
\end{array}
\]
Since $\Vert A\Vert\Vert B\Vert^{2}<2^{p}(\frac{1}{2^{p}})=1$, $C_{T}=C_{\int_{E}k(t,s,x(s))ds + g(t)}\neq \emptyset$ and $(X, L(H), d_{b})$ has the properties $(P3)$ and $(P4)$, then by Corollary \ref{c3.4} the integral equation (\ref{eq4.9}) has a unique solution $x^{*}$ in $L^{\infty}(E)$.

\end{proof}


\end{document}